\documentclass[11pt,reqno]{article}
\usepackage{graphicx}
\usepackage{amssymb}
\usepackage{amsmath}
\usepackage{amsthm}
\usepackage{amsfonts}
\usepackage{url}
\usepackage{hyperref}
\usepackage{breakurl}
\usepackage[T1]{fontenc}        
\usepackage[usenames]{color}

\DeclareMathOperator{\dup}{d\hspace{-1.5pt}}

\newcommand{\comment}[1]{}

\newcommand{\raisecomma}{\raisebox{2pt}{$,$}}
\newcommand{\raisedot}{\raisebox{2pt}{$.$}}

\newcommand{\tm}[1]{{\times 10^{#1}}}

\newcommand{\That}{{\widehat{T}}}
\newcommand{\Rhat}{{\widehat{R}}}
\newcommand{\Ttilde}{{\widetilde{T}}}
\newcommand{\Rtilde}{{\widetilde{R}}}

\newcommand{\half}{{\textstyle\frac12}}

\newcommand{\C}{{\mathbb C}}
\newcommand{\R}{{\mathbb R}}
\newcommand{\N}{{\mathbb N}}
\newcommand{\Nstar}{{\N^{*}}}

\newcommand{\Z}{{\mathbb Z}}

\newcommand{\Hplane}{{\mathcal H}}
\newcommand{\Hplanestar}{{\Hplane^{*}}}

\newcommand{\vth}{\vartheta}

\newcommand{\logGamma}{\ln\Gamma}

\newcommand{\Tableck}{{$1$}}	
\newcommand{\Tableabcd}{{$2$}}	

\newtheorem{theorem}{Theorem}
\newtheorem{corollary}{Corollary}
\newtheorem{lemma}{Lemma}

\newtheorem{remark}{Remark}

\begin{document}
\bibliographystyle{plain}
\title{~\\[-60pt]
On
asymptotic approximations to the log-Gamma and
Riemann-Siegel theta functions
}
\author{Richard P.\ Brent\\[5pt]
Australian National University\\
Canberra, ACT 2600,
Australia\\[5pt]
and CARMA\\
University of Newcastle\\
Callaghan, NSW 2308, Australia\\
}

\date{\today\\[20pt]
In memory of Jonathan Borwein 1951--2016} 

\maketitle
\thispagestyle{empty}                   

\begin{abstract}

We give bounds on the error in the asymptotic approximation of the
log-Gamma function $\ln\Gamma(z)$
for complex $z$ in the right half-plane. These improve on earlier bounds
by Behnke and Sommer (1962),  Spira (1971), and Hare (1997).
We show that 
$|R_{k+1}(z)/T_k(z)| < \sqrt{\pi k}$ 
for nonzero $z$ in the right half-plane, where $T_k(z)$ is
the $k$-th term in the asymptotic series, and $R_{k+1}(z)$ is the error
incurred in truncating the series after $k$ terms.
If $k \le |z|$, then the stronger bound
$|R_{k+1}(z)/T_k(z)| < (k/|z|)^2/(\pi^2-1)
< 0.113$
holds. Similarly for the asymptotic approximation
of $\ln\Gamma(z+\frac{1}{2})$, except that a factor
$\eta_k = 1/(1-2^{1-2k})$ multiplies some of the bounds.

We deduce similar bounds for asymptotic approximation of the
Riemann-Siegel theta function $\vartheta(t)$. We show that the 
accuracy of a well-known approximation to $\vartheta(t)$ can be improved by
including an 
\hbox{exponentially} small term in the approximation.
This improves the \hbox{attainable} accuracy for real $t>0$
from $O(\exp(-\pi t))$ to $O(\exp(-2\pi t))$.
We discuss a similar example due to Olver (1964), and
a connection with the Stokes phenomenon.
\end{abstract}

\pagebreak[4]

\section{Introduction}			\label{sec:intro}

The \emph{Riemann-Siegel theta function} $\vth(t)$,
which occurs in the theory of the 
Riemann zeta function~\cite[\S6.5]{Edwards},
is defined for real $t$ by
\begin{equation}		\label{eq:RS-theta-1}
\vth(t) :=
\arg\Gamma\!\left(\frac{it}{2}+\frac{1}{4}\right) -
  \frac{t}{2}\log\pi.
\end{equation}
The argument is defined so that $\vth(t)$ is continuous
on $\R$, and $\vth(0) = 0$.
Clearly $\vth(t)$ is an odd function,
i.e. $\vth(-t) = -\vth(t)$ for all real~$t$,
so there is no essential loss of generality in
assuming that $t$ is positive.

The significance of $\vth(t)$ is the fact that
$
Z(t) := \exp(i\vth(t))\,\zeta(\half+it)
$ 
is a real-valued
function.  Thus, zeros of $\zeta(s)$ on the critical line
$\Re(s) = \half$ can be detected
by sign changes of $Z(t)$.
In a sense, $\vth(t)$ encodes half the information contained in
$\zeta(\half+it)$ (albeit the less interesting half), while
$Z(t)$ encodes the other half.

The motivation for this paper was an attempt to give a straight-forward
proof for the well-known asymptotic expansion
\begin{equation}		\label{eq:old2.3_corrected}
\vth(t) \sim \frac{t}{2} \log\left(\frac{t}{2\pi e}\right)
 - \frac{\pi}{8}
 + \sum_{j=1}^{\infty}\frac{(1-2^{1-2j})\,|B_{2j}|}{4j(2j-1)\,t^{2j-1}}
 \,\raisecomma
\end{equation}
and to obtain a rigorous bound on the error incurred in truncating the
sum after $k$ terms.  A bound 
\begin{equation}		\label{eq:faulty_bound}
\frac{(2k)!}{(2\pi)^{2k+2}\,t^{2k+1}} + \exp(-\pi t)
\end{equation}
was stated
in~\cite[following eqn.~(2.3)]{rpb047}, but no proof was given, and in
fact the bound is incorrect.\footnote{We have taken into account a
typographical error in eqn.~(2.3), where $B_{2k}$ should be replaced
by $|B_{2k}|$, as previously noted in~\cite[footnote on pg.~682]{rpb070}.
}
For example, with $k=3$ and $t=9.5$, the
error exceeds the bound by a factor of $1.011$.

To obtain a satisfactory error bound to replace~\eqref{eq:faulty_bound}
we needed an error bound
for Stirling's asymptotic
approximation~\cite[(6.1.40)]{AS} 
to $\logGamma(z)$ on the
imaginary axis $\Re(z)=0$. We found several such bounds in the literature,
but they were not entirely satisfactory for our purposes
(see Remarks~$\ref{remark:Spira}$--$\ref{remark:Stieltjes}$).
Hence, Theorems~\ref{thm:right_half_bd}--\ref{thm:iterative}
and Corollary~\ref{cor:cor1} give new
error bounds on Stirling's approximation. These bounds are valid in the
right half-plane ($\Re(z) \ge 0$, $z\ne 0$), and improve on
previous bounds  when $z$ is
on or sufficiently close to the imaginary axis.

Stirling's approximation leads, via the duplication formula for the Gamma
function, to an asymptotic expansion
\[
\logGamma(z+\half) \sim z\log z - z + \half\log(2\pi) +
  \sum_{j=1}^{\infty}\frac{B_{2j}(\frac12)}{2j(2j-1)\,z^{2j-1}}
\]
that goes back to Gauss \cite[Eqn.\ {[}59{]} of Art.\ 29]{Gauss-v3}.
It is the special case $a=\half$ of an expansion for
$\logGamma(z+a)$ that was considered, for $a\in[0,1]$ and real positive $z$,
by Hermite~\cite{Hermite}.
See also Askey and Roy~\cite[5.11.8]{AR}, and Nemes~\cite[(1.6)]{Nemes13}.
Using our bounds on the error in Stirling's approximation to
$\logGamma(z)$, we deduce bounds on the error in Gauss's 
approximation to $\logGamma(z+\half)$.
The bounds are almost the same as those for Stirling's approximation,
the only difference being that a factor $\eta_k = 1/(1-2^{1-2k})$
multiplies some of the bounds
(see Theorems~\ref{thm:right_half_bd_a}--\ref{thm:iterative2}
and Corollary~\ref{cor:cor1hat} in \S\ref{sec:Gamma2}).

These bounds, in the case
that $z=it$ ($t\in\R$), are what is needed to give bounds on the
approximation of $\vth(t)$.  
See Theorem~\ref{thm:RS-theta-approx_bds} and
Corollaries~\ref{cor:cor3}--\ref{cor:cor4} in \S\ref{sec:RS-theta} for
these bounds.
One such result (see \eqref{eq:bd_no_arctan} below) is a bound
\begin{equation}		\label{eq:old_bd_corrected}
\eta_k\,(\pi k)^{1/2}\,\Ttilde_k(t) + \half e^{-\pi t}
\end{equation}
on the error if
the sum in~\eqref{eq:old2.3_corrected} is truncated
after the $k$-th term $\Ttilde_k(t)$.

Perhaps surprisingly, we obtain a smaller bound
if an exponentially-small term
$\half\arctan(\exp({-\pi t}))$ is included in the approximation
of $\vth(t)$.
The term $\half \exp({-\pi t})$ in~\eqref{eq:old_bd_corrected} can 
then be omitted (see Theorem~\ref{thm:RS-theta-approx_bds}).
This is discussed in \S\S\ref{sec:RS-theta}--\ref{sec:accuracy}.
In \S\ref{sec:accuracy} we show that the
attainable error, if the terms in the asymptotic series are summed until
the smallest term is reached, is of order $\exp({-\pi t})$
if (as usual) the $\arctan$ term is omitted from the approximation,
but is reduced to $O(\exp({-2\pi t}))$ if the
$\arctan$ term is included.
This observation is to some extent implicit in the work of
Berry~\cite[\S4]{Berry95} and Gabcke~\cite[Satz~4.2.3]{Gabcke}, 
but our presentation
makes it explicit.\footnote{The fact that the error in
the Riemann-Siegel approximation to $Z(t)$ is of order $\exp(-\pi t)$
was observed empirically by the author in 1977, when writing a review
of~\cite{CR}. A detailed theoretical explanation was later given by
Berry~\cite{Berry95}.
}

\section[Asymptotic approximation of logGamma(z)]%
{Asymptotic approximation of $\logGamma(z)$} \label{sec:Gamma1}

A comment on notation: variables
$s, z \in\C$; 
$c, r, t, u, x, y, \varepsilon, \eta, \theta, \psi \in\R$;
and $j, k, m, n\in \Nstar$ (the positive integers).
``log'' denotes the
principal branch of the natural logarithm on the cut plane
$\C\backslash(-\infty,0]$.
The (closed) right half-plane is
$\Hplane := \{z\in\C: \Re(z)\ge 0\}$,
and $\Hplanestar := \Hplane\backslash\{0\}$.
We define constants $\eta_k$ for $k\in\Nstar$  by
$\eta_k := 1/(1-2^{1-2k})$.

The proper domain for the log-Gamma function $\logGamma$ is a Riemann
surface. However, for our purposes it is sufficient to take the
(principal branch of the) log-Gamma function to be an analytic function
on the cut-plane $\C\backslash(-\infty,0]$, such that
$\logGamma(x) = \log(\Gamma(x))$ is real for positive real $x$.\footnote{In
a computer implementation of $\logGamma(z)$,
care has to be taken because $\logGamma(z)$ and $\ln(\Gamma(z))$
may differ by a multiple of $2\pi i$.}

In this section we consider approximation of $\logGamma(z)$ for
$z\in\C\backslash(-\infty,0]$.
When computing $\Gamma(z)$ or $\logGamma(z)$, 
we can use the reflection formula
\[
\Gamma(z)\Gamma(-z) = -\frac{\pi}{z\sin(\pi z)}
\]
if $\Re(z) < 0$, $z\not \in \Z$. 
Thus, in the following we assume that $\Re(z) \ge 0$.

We recall Stirling's approximation,
taking $k-1$ terms in the asymptotic expansion with
a remainder $R_k$:
\begin{equation}		\label{eq:Stirling}
\logGamma(z) = (z-\half)\log z - z + \half\log(2\pi)
	+ \sum_{j=1}^{k-1} T_j(z) + R_k(z),
\end{equation}
where
\begin{equation}		\label{eq:T_k}
T_j(z) = \frac{B_{2j}}{2j(2j-1)z^{2j-1}}\,\raisecomma
\end{equation}
and $R_k(z)$ is a ``remainder'' or ``error'' term that may be written as
\begin{equation}		\label{eq:Stirling_R}
R_k(z) = 
\int_0^\infty \frac{B_{2k} - B_{2k}(\{u\})}
	{2k\,(u+z)^{2k}}\dup u.
\end{equation}
Here $\{u\} := u - \lfloor u\rfloor$
denotes the fractional part of $u$,
$B_{2k}(u)$ is a Bernoulli polynomial, and $B_{2k} = B_{2k}(0)$ is
a Bernoulli number, so $B_2 = \frac{1}{6}$,
$B_4 = -\frac{1}{30}$, etc.
See Olver~\cite[\S\S8.1, 8.4]{Olver} for the definitions and
a proof of~\eqref{eq:Stirling_R}.

A different representation of the remainder is often
convenient.  Using~\eqref{eq:Stirling_R} 
and $R_{k}(z) = T_k(z) + R_{k+1}(z)$, we see that the error after
taking $k$ terms (instead of $k-1$) in the sum is%
\footnote{We have followed
Olver's convention.
Other authors may include $k$ terms in the sum
in~\eqref{eq:Stirling}. Thus, their $R_k$
may correspond to our $R_{k+1}$, and care has to be taken when comparing
bounds in the literature. See, for example, 
Abramowitz and Stegun~\cite[(6.1.42)]{AS}.}
\begin{equation}		\label{eq:StirlingR2}
R_{k+1}(z) = -\int_0^\infty \frac{B_{2k}(\{u\})}
        {2k\,(u+z)^{2k}}\dup u.
\end{equation}

If $z$ is real and positive, 
then the asymptotic series~\eqref{eq:Stirling} is strictly
enveloping in the sense of P\'olya and Szeg\"o~\cite[Ch.~4]{PS-v1}, 
so $R_k(z)$ has the same sign as the first term omitted, which is $T_k(z)$.
Also, $R_k(z)$ is smaller in magnitude than this term,
i.e.~$|R_k(z)| < |T_k(z)|$ 
(in fact this inequality holds whenever $|\arg(z)| \le\pi/4$,
see Remark~\ref{remark:WW}).

In the case of complex $z$ in the right half-plane, 
the error $R_k(z)$ may be larger in absolute value
than the first omitted term. This case is covered by
Theorem~\ref{thm:right_half_bd} and Corollary~\ref{cor:cor1},
which improve on earlier results
by Spira~\cite{Spira} and Hare~\cite[Prop.~4.1]{Hare}.

\begin{theorem}			\label{thm:right_half_bd}
If $z\in\Hplanestar$,
$R_k(z)$ is defined by
eqn.~\eqref{eq:Stirling}, 
and $T_j(z)$ by~\eqref{eq:T_k},
then
\begin{equation}		\label{eq:right_bd1}
|R_{k+1}(z)| \le
	\frac{\pi^{1/2}\,\Gamma(k+\half)}{\Gamma(k)}\; |T_k(z)|
\end{equation}
and
\begin{equation}		\label{eq:right_bd2}
|R_{k}(z)| \le
	\left(\frac{\pi^{1/2}\,\Gamma(k+\half)}{\Gamma(k)} + 1\right)
	|T_k(z)|.
\end{equation}
\end{theorem}
\begin{proof}
Let $x = \Re(z)$ and $y = \Im(z)$. From~\eqref{eq:StirlingR2}, we have
\begin{equation}		\label{eq:integrals}
|R_{k+1}(z)|
 = \left|\int_0^\infty\frac{B_{2k}(\{u\})}{2k(u+z)^{2k}}\dup u\right|
 \le \frac{|B_{2k}|}{2k}\int_0^\infty|u+z|^{-2k}\dup u.
\end{equation}
Since $x\ge 0$, inside the integral we have that
\[
|u+z|^2 = (u+x)^2+y^2 \ge u^2+x^2+y^2 = u^2+|z|^2.
\]
Making a change of variables $u \mapsto |z|\tan\psi$, this gives
\begin{align*}
\int_0^\infty|u+z|^{-2k}\dup u
 \le& \int_0^\infty(u^2+|z|^2)^{-k}\dup u\\
 =&\; |z|^{1-2k}\int_0^{\pi/2}\cos^{2k-2}\psi\dup\psi\\
 =&\; \frac{\pi^{1/2}}{2}\,\frac{\Gamma(k-\half)}{\Gamma(k)}\,|z|^{1-2k},
\end{align*}
where the closed form for the integral is known as ``Wallis's formula'',
see for example~\cite[(6.1.49)]{AS}.  Thus, the inequality
\eqref{eq:right_bd1} follows from~\eqref{eq:integrals}.

The inequality~\eqref{eq:right_bd2}
follows easily from~\eqref{eq:right_bd1} and the triangle inequality
\begin{equation}		\label{eq:convert_bd}
|R_k(z)| = |T_k(z) + R_{k+1}(z)| \le |T_k(z)| + |R_{k+1}(z)|.
\end{equation}
\end{proof}
During a computation, we may wish to bound the error term as a multiple of
either the last term included in the approximating sum, or as a multiple
of the first term omitted.  Hence, the following corollary of
Theorem~\ref{thm:right_half_bd} is useful.
\begin{corollary}		\label{cor:cor1}
If $z\in\Hplanestar$ and $R_k(z)$ is defined by
eqn.~\eqref{eq:Stirling}, then
\begin{equation}		\label{eq:corbd1}
\left|\frac{R_{k+1}(z)}{T_k(z)}\right| < \sqrt{\pi k}
\end{equation}
and
\begin{equation}		\label{eq:corbd2}	
\left|\frac{R_{k}(z)}{T_k(z)}\right| < 1+ \sqrt{\pi k}.
\end{equation}
\end{corollary}
\begin{proof}
{From}~\cite[eqn.~(21)]{rpb267},
\[
\logGamma(x+\half)-\logGamma(x) - \half\log(x)
 \sim - \frac{1}{8x} + \cdots,
\]
where the asymptotic series on the right
is strictly enveloping for positive real~$x$.
Thus, we have
$\log(\Gamma(x+\half)/\Gamma(x)) < \half\log x$, which implies
that $\Gamma(k+\half)/\Gamma(k) < \sqrt{k}$.
The inequality~\eqref{eq:corbd1} now follows from \eqref{eq:right_bd1} of
Theorem~\ref{thm:right_half_bd} and the
definition of $T_k(z)$.
The inequality~\eqref{eq:corbd2} follows similarly, from~\eqref{eq:right_bd2}
of Theorem~\ref{thm:right_half_bd}, or directly from~\eqref{eq:convert_bd}.
\end{proof}
\begin{remark}				\label{remark:convert_bd}
{\rm
The device of converting a bound on $R_{k+1}(z)$ into a bound on
$R_k(z)$, of the same order in $|z|$, via the
triangle inequality~\eqref{eq:convert_bd}, also applies to the bounds
given in \S\S\ref{sec:Gamma2}--\ref{sec:RS-theta} below. For the sake
of brevity we do not always give such bounds explicitly.
}
\end{remark}

In Remarks~\ref{remark:Spira}--\ref{remark:Stieltjes}
we comment briefly on some
related bounds that may be found in the literature, allowing
for different notations. Here and elsewhere, 
we define $\theta = \theta(z) := \arg z$.
\begin{remark}				\label{remark:Spira}
{\rm
Spira~\cite[eqn.~(4)]{Spira} obtains a bound of the same form
as our~\eqref{eq:right_bd1}, but larger by a factor
of approximately $4\sqrt{k/\pi}$. This is primarily
because he uses a rather crude upper bound on
the relevant integral instead of using 
Wallis's formula.\footnote{We note that the proof given by Spira
\cite[top of page 319]{Spira} is incomplete~-- he only proves a bound 
of the form $c(k)/|\Im(z)|^{2k-1}$, not the claimed $c(k)/|z|^{2k-1}$.}
}
\end{remark}
\begin{remark}				\label{remark:Hare}
{\rm
Hare~\cite[Prop.~4.1]{Hare}
obtains a bound of the form $c(k)/|\Im(z)|^{2k-1}$, assuming that
$\Im(z) \ne 0$, but without the assumption
that $\Re(z) \ge 0$.
Here $c(k) = 4\pi^{1/2}\Gamma(k+\half)/\Gamma(k) \sim 4\sqrt{\pi k}$.
When both bounds are applicable, 
our bound~\eqref{eq:right_bd2} is better
by a factor of about $4/|\sin\theta|^{2k-1}$ (for large~$k$).
A problem with a bound involving $|\Im(z)|$ rather than $|z|$ is that
the bound can not be reduced by applying the recurrence
$\Gamma(z+1) = z\Gamma(z)$.
}
\end{remark}
\begin{remark}				\label{remark:BS}
{\rm
In Behnke and Sommer~\cite[(18) on pg.~304]{BS} we find a bound that (in our
notation) is
\begin{equation}			\label{eq:BS}
\left|\frac{R_{k+1}(z)}{T_{k+1}(z)}\right| <
 1 + \frac{2k+1}{2}\sqrt{\frac{\pi}{k}}\,\raisecomma
\end{equation}
valid for $k \ge 1$ and $\Re(z) \ge 0$, $z \ne 0$.
It is interesting to note that this predates the bounds of
Spira~\cite{Spira} and Hare~\cite{Hare}.
To compare with our bounds, make a change of
variables $k \mapsto k+1$ in~\eqref{eq:corbd2} to obtain
\begin{equation}			\label{eq:corbd2plus}
\left|\frac{R_{k+1}(z)}{T_{k+1}(z)}\right| <
 1 + \sqrt{\pi(k+1)}.
\end{equation}
Since $k+1 < (k+\half)^2/k$, our bound~\eqref{eq:corbd2plus}
is always smaller than Behnke and Sommer's bound~\eqref{eq:BS},
although the ratio tends to $1$ as $k \to \infty$.  Note that our
bound~\eqref{eq:corbd2} gives a valid bound $1+\sqrt{\pi}$ on 
$|R_1(z)/T_1(z)|$, whereas~\eqref{eq:BS} requires $k \ge 1$ 
as the right-hand side is undefined if $k=0$.
}
\end{remark}
\begin{remark}				\label{remark:WW}
{\rm
A bound due to
Whittaker and Watson~\cite[pg.~252]{WW} (see also \cite[(6.1.42)]{AS}),
valid for $\Re(z) > 0$, is:
\begin{equation}			\label{eq:WWbound}
|R_k(z)| \le K(z)\,|T_k(z)|,
\end{equation}
where $K(z) = \sup_{u \ge 0} |z^2/(u^2+z^2)|$.
It is easy to see that $K(z)$ depends only on $\theta(z)$.
A geometric argument shows that
\[
K(z) = \begin{cases}
	 1 \text{ if } |\theta|  \le \pi/4;\\
	 \displaystyle\frac{1}{|\sin(2\theta)|} 
	  \text { if } |\theta| \in (\pi/4, \pi/2).
	\end{cases}
\]
Thus, the bound~\eqref{eq:WWbound} is preferable to those mentioned
in Remarks~\ref{remark:Spira}--\ref{remark:BS} 
(and to our bound~\eqref{eq:corbd2})
if $|\theta| \le \pi/4$, but it becomes poor as $|\theta|$ approaches
$\pi/2$.
}
\end{remark}
\begin{remark}				\label{remark:Stieltjes}
{\rm
A bound due to Stieltjes (see Olver~\cite[(8.4.06)]{Olver}) is
\begin{equation}		\label{eq:Olver84}
|R_k(z)| \le |T_k(z)|\sec^{2k}(\theta/2),
\end{equation}
where $|\theta| < \pi$.
This differs from our bound~\eqref{eq:right_bd2}
by a factor of approximately $\sec^{2k}(\theta/2)/\sqrt{\pi k}$.
If $\theta \approx \pi/2$ this factor is approximately $2^k/\sqrt{\pi k}$,
which is greater than~$1$ for all $k\ge 1$.
Thus, \eqref{eq:Olver84} is better
than our bound only if $|\theta|$ is sufficiently small. However, if
$|\theta| \le \pi/4$ we should prefer the bound~\eqref{eq:WWbound}.
}
\end{remark}

It is natural to ask if an upper bound of order $k^{1/2}$
for $|R_{k+1}(z)/T_k(z)|$, as in Corollary~\ref{cor:cor1}, is the
best possible.  Certainly, when $|\arg(z)| \le \pi/4$, or when
$|T_k(z)|$ is much larger than $|T_{k+1}(z)|$, the bound is not optimal.
However, without imposing conditions on $k$ and/or $z$, the bounds of 
Corollary~\ref{cor:cor1} are the best possible, up to constant factors.
We sketch a proof of this. Let $n$ be a sufficiently large positive integer, 
and $z = iy$, where $y = n/\pi$.  Thus, $n$ is close to the index
of the minimal term $|T_j(z)|$.  Also, there is no cancellation in
the sum $T_1(z) + T_2(z) + \cdots + T_n(z)$, since, using~\eqref{eq:T_k},
\[
i\,T_j(iy) = \frac{i\,(-1)^{j-1}|B_{2j}|}{2j(2j-1)\,(iy)^{2j-1}}
	 = \frac{|B_{2j}|}{2j(2j-1)\,y^{2j-1}} 
\]
is real and positive.
Using Stirling's approximation to estimate $T_j(z)$ and $T_n(z)$, 
we have
\[
\left|\frac{T_j(z)}{T_n(z)}\right| = 
 1 + O\left(\frac{\delta^2}{y}\right),
\]
if $j = n-\delta$ and $\delta^2 \le y$.
Thus, we can choose a positive integer
$\delta$ of order $y^{1/2}$ so that
$1/2 \le |T_j(z)/T_n(z)| \le 2$ 
for $n-\delta \le j \le n$. 
Hence $|T_{n-\delta}(z)+\cdots+T_{n-1}(z)| \ge 
 \delta\,|T_n(z)|/2$. 
For some $k$ in the interval $[n-\delta,n]$, we
must have $|R_{k+1}(z)/T_n(z)| \ge \delta/4$, so
$|R_{k+1}(z)/T_k(z)| \ge \delta/8$ is of order~$y^{1/2}$.

Numerical evidence confirms this conclusion.
Taking $n=100$, $y=n/\pi$, 
and $k=90$,
we find that $|R_{k+1}(iy)/T_k(iy)| \approx 4.62$.
If $n=400$, $y=n/\pi$, $k=383$, then $|R_{k+1}(iy)/T_k(iy)| \approx 10.15$. 
Thus, it appears that the constant $\sqrt{\pi}$ appearing
in Corollary~\ref{cor:cor1} can not be reduced by a factor greater
than~$4$ when $z$ lies on, or sufficiently close to,%
\footnote{The proof that we have outlined can be modified to cover
a region of the form $\Re(z) \ge 0$, $|\Im(z)| \ge c\Re(z)^2$, 
where $c$ is a sufficiently large positive constant.  On the other hand, by
Whittaker and Watson's bound~\eqref{eq:WWbound},
it can not be extended into the
sector $|\theta| < \pi/2 - \varepsilon$ ($|z|$ sufficiently large),
since in that region $|R_k(z)/T_k(z)|$ and $|R_{k+1}(z)/T_k(z)|$ are
$O(1/\varepsilon)$.
} the imaginary axis.

In Theorem~\ref{thm:iterative}, we obtain bounds that are better
than the bounds given in Theorem~\ref{thm:right_half_bd}
and Corollary~\ref{cor:cor1},
provided the condition $k \le |z|$ is satisfied.
If $|z|$ is too small, we can apply the recurrence 
$\logGamma(z) = \logGamma(z+1) - \log z$
as often as necessary and then apply Theorem~\ref{thm:iterative}.

Before stating Theorem~\ref{thm:iterative}, we define some constants $c_k$
which enter into the proof of the theorem.
Assuming that $T_k(z)$ is defined by~\eqref{eq:T_k}, let
\[
c_k := 
 \sum_{j=1}^{2k} \left|\frac{T_{k+j}(k)}{T_k(k)}\right|
	+ \sqrt{3k\pi}\left|\frac{T_{3k}(k)}{T_k(k)}\right|.
\]
The following lemma is the reason for introducing the constants $c_k$.
\begin{lemma}			\label{lem:ck}
If $z\in\Hplanestar$, $R_k(z)$ is defined by
eqn.~\eqref{eq:Stirling}, and $k \le |z|$, then
\begin{equation}                \label{eq:corbd3aa}
\left|\frac{R_{k+1}(z)}{T_k(z)}\right| \le c_k\,({k}/{|z|})^2.
\end{equation} 
\end{lemma}
\begin{proof}
For all $m\in\N$,
\begin{equation}		\label{eq:Rkm1}
R_{k+1}(z) = \sum_{j=1}^m T_{k+j}(z) + R_{k+m+1}(z).
\end{equation}
Now
\[
|R_{k+m+1}(z)| \le \sqrt{(k+m)\pi}\,|T_{k+m}(z)|,
\]
by Corollary~\ref{cor:cor1} with $k$ replaced by $k+m$.
Taking norms in~\eqref{eq:Rkm1}, choosing $m=2k$, and dividing both
sides by $|T_{k+1}(z)|$, we obtain
\[
\left|\frac{R_{k+1}(z)}{T_{k+1}(z)}\right| \le
  \frac{1}{|T_{k+1}(z)|}\left(\sum_{j=1}^{2k} |T_{k+j}(z)| + 
    \sqrt{3k\pi}\,|T_{3k}(z)|\right).
\]
Since $|T_{k+j}(z)/T_{k+1}(z)|$ has the form $c/|z|^{2j-2}$,
it is a non-increasing function of $|z|$ (assuming $j \ge 1$), so its
maximum occurs when $|z|$ is minimal, i.e.~when $|z| = k$.
Thus
\[
\left|\frac{R_{k+1}(z)}{T_{k+1}(z)}\right| \le
  \frac{1}{|T_{k+1}(k)|}\left(\sum_{j=1}^{2k} |T_{k+j}(k)| +
    \sqrt{3k\pi}\,|T_{3k}(k)|\right)
 = c_k\left|\frac{T_k(k)}{T_{k+1}(k)}\right|.
\]
Since $T_{k+1}(z)/T_{k}(z)$ has the form $c/z^2$, we have
\[
\left|\frac{T_{k+1}(z)}{T_k(z)}\right| = (k/|z|)^2 
	\left|\frac{T_{k+1}(k)}{T_k(k)}\right|.
\]
Thus,~\eqref{eq:corbd3aa} follows.
\end{proof}

Numerical values of $c_k$ for various $k\le 50$ are given in
Table~{\Tableck}. The $c_k$ appear to increase monotonically
to the limit $1/(\pi^2-1) \approx 0.112745$.
We have verified monotonicity, and that $c_k < 1/(\pi^2-1)$,
for $k \le 100$.

\begin{table}[ht]
\begin{center}
\begin{tabular}{cc|cc|cc}
$k$  	& $c_{k}$ 	& $k$ 	& $c_{k}$ 	& $k$ 	& $c_{k}$\\
\hline
$1$	& $0.072096$	& $6$	& $0.107384$	& $15$	& $0.110498$\\
$2$	& $0.103961$	& $7$	& $0.108089$	& $20$	& $0.111050$\\
$3$	& $0.104294$	& $8$	& $0.108634$	& $25$	& $0.111384$\\
$4$	& $0.105304$	& $9$	& $0.109067$	& $30$	& $0.111609$\\
$5$	& $0.106460$	& $10$	& $0.109419$	& $50$	& $0.112060$\\
\hline
\end{tabular}
\caption{The constants $c_k$ (rounded up to $6$ decimals).}
\end{center}
\vspace*{-10pt}
\end{table}

\pagebreak[3]

\begin{theorem}			\label{thm:iterative}
If $z\in\Hplanestar$, $R_k(z)$ is defined by
eqn.~\eqref{eq:Stirling}, and $k \le |z|$, then
\begin{equation}                \label{eq:corbd3a}
\left|\frac{R_{k+1}(z)}{T_k(z)}\right|
 < \frac{(k/|z|)^2}{\pi^2-1}
 \le \frac{1}{\pi^2-1} < 0.113
\end{equation} 
and
\begin{equation}                \label{eq:corbd3b}
\left|\frac{R_{k}(z)}{T_k(z)}\right|
 < 1 + \frac{(k/|z|)^2}{\pi^2-1}
 \le \frac{\pi^2}{\pi^2-1} < 1.113.
\end{equation} 
\end{theorem}

\begin{proof}[Proof of Theorem~$\ref{thm:iterative}$.]
Let
\[
\mu := \left(\frac{k}{\pi|z|}\right)^2 \le \frac{1}{\pi^2}
\]
and $m := \lfloor k^{1/2}\rfloor$.
For brevity, we write $R_k$ for $R_k(z)$ and $T_k$ for $T_k(z)$.
Since $R_{k+1} = T_{k+1} + T_{k+2} + \cdots + T_{k+m} + R_{k+m+1}$,
we have $|R_{k+1}/T_k| \le S+E$, where
\[S := \sum_{j=1}^m \left|\frac{T_{k+j}}{T_k}\right|
 \;\text{ and }\;
  E := \left|\frac{R_{k+m+1}}{T_k}\right|.
\]
Since $|B_{2k}| = 2(2k)!\,\zeta(2k)/(2\pi)^{2k}$,
we have 
\[
\left|\frac{T_{k+j}}{T_k}\right| \le
 \frac{(2k+2j-2)!}{(2k-2)!}\,|2\pi z|^{-2j}.
\]
Using the assumption $k\le|z|$, it follows that
\begin{equation}			\label{eq:Tk_prod_bd}
\left|\frac{T_{k+j}}{T_k}\right| \le
 \mu^j \prod_{n=1}^{2j}\left(1+\frac{n-2}{2k}\right).
\end{equation}
Now $1+x \le \exp(x)$ for all $x\in\R$.  Thus
\[
\left|\frac{T_{k+j}}{T_k}\right| \le
 \mu^j \prod_{n=1}^{2j}\exp\left(\frac{n-2}{2k}\right)
 = \mu^j \exp\left(\frac{(2j-3)j}{2k}\right).
\]
By convexity, $1\le\exp(x)\le 1+(e-1)x$ for all $x\in [0,1]$.
It follows that, for $2 \le j \le m$, we have
\begin{equation}		\label{eq:Tkratj}
\left|\frac{T_{k+j}}{T_k}\right| \le
  \mu^j \left(1+(e-1)\frac{(2j-3)j}{2k}\right).
\end{equation}
Also, for the special case $j=1$, the inequality~\eqref{eq:Tk_prod_bd} gives
\begin{equation}		\label{eq:Tkrat1}
\left|\frac{T_{k+1}}{T_k}\right| \le
  \mu \left(1-\frac{1}{2k}\right).
\end{equation}
{From}~\eqref{eq:Tkratj}--\eqref{eq:Tkrat1},
\begin{align}
S &\le -\frac{\mu}{2k} + \sum_{j=1}^m \mu^j
   + \frac{e-1}{2k}\sum_{j=2}^m (2j-3)j\,\mu^j\nonumber\\
  &< -\frac{\mu}{2k} + \sum_{j=1}^\infty \mu^j
   + \frac{e-1}{2k}\sum_{j=2}^\infty (2j-3)j\,\mu^j\nonumber\\
  &= -\frac{\mu}{2k} + \frac{\mu}{1-\mu}
	+ \left(\frac{e-1}{2k}\right)
	   \frac{\mu^2(2+3\mu-\mu^2)}{(1-\mu)^3}\,\raisedot
							\label{eq:Sineq1}
\end{align}
Thus
\[
\frac{\mu}{1-\mu} - S \;>\; \frac{\mu}{2k}\left[1 -
    \frac{(e-1)\mu(2+3\mu-\mu^2)}{(1-\mu)^3}\right].
\]
Since $\mu(2+3\mu-\mu^2)/(1-\mu)^3 = \sum_{j=2}^\infty (2j-3)j\mu^{j-1}$ 
is monotonic increasing on $[0,1/\pi^2]$, the factor in square
brackets attains its minimum on $[0,1/\pi^2]$ at $\mu = 1/\pi^2$, and a 
numerical computation shows that the minimum is greater than $\pi^2/22$.
Thus,
\[
\frac{\mu}{1-\mu} - S > \frac{\pi^2\mu}{44k}\,\raisedot
\]

Now consider $E$. We have
\[
E = \left|\frac{R_{k+m+1}}{T_k}\right|
    = \left|\frac{T_{k+m}}{T_k}\right|\cdot
      \left|\frac{R_{k+m+1}}{T_{k+m}}\right|\,\raisedot
\]
The first factor on the right is at most $\mu^m e$, by~\eqref{eq:Tkratj}
with $j=m$;
the second factor is at most $\sqrt{\pi(k+m)}$, by an application
of Corollary~\ref{cor:cor1} with $k$ replaced by $k+m$.   This gives
\[
E \le \mu^m e\sqrt{\pi(k+m)}
  \le \mu^{\sqrt{k}-1} e\sqrt{2\pi k}.
\]
Thus  $kE/\mu \le \mu^{\sqrt{k}-2}e\sqrt{2\pi k^3} \ll 1/k$,
so there exists $k_0$ such that, for all $k \ge k_0$,
$kE/\mu < \pi^2/44$, so $E < \pi^2\mu/(44k)$ and $\mu/(1-\mu) > S+E$.
A computation shows that we can take 
$k_0 = 34$.
Thus, for all $k \ge k_0$,
\[
\left|\frac{R_{k+1}}{T_k}\right|
 < \frac{\mu}{1-\mu} 
 = \frac{k^2}{\pi^2|z|^2-k^2}			
 \le \frac{(k/|z|)^2}{\pi^2-1}\,\raisedot	
\]
This proves the desired inequality~\eqref{eq:corbd3a} for $k \ge k_0$.

By a straightforward numerical computation, we can verify
that~\eqref{eq:corbd3a} also
holds for $1 \le k \le 33$ (see Lemma~\ref{lem:ck} and Table \Tableck).
This concludes the proof of~\eqref{eq:corbd3a}.
Finally, \eqref{eq:corbd3b} follows from~\eqref{eq:corbd3a}
and the triangle inequality.
\end{proof}

\begin{remark}		\label{remark:conjectured_improvement}
{\rm
It is reasonable to conjecture the slightly stronger inequalities
\begin{equation}	\label{eq:conjectured_improvement}
\left|\frac{R_{k+1}(z)}{T_k(z)}\right|
 < \frac{k^2}{\pi^2|z|^2-k^2}, \;\;
\left|\frac{R_{k}(z)}{T_k(z)}\right|
 < \frac{\pi^2|z|^2}{\pi^2|z|^2-k^2},
\end{equation}
for all $(k, z)$ such that $|z| \ge k \ge 1$.
This has been verified numerically, and the proof
of Theorem~\ref{thm:iterative} shows that
\eqref{eq:conjectured_improvement} holds for $k \ge 34$.
However, our proof of~\eqref{eq:corbd3a} for $k \le 33$,
using Lemma~\ref{lem:ck} and the constants $c_k$,
is insufficient to prove~\eqref{eq:conjectured_improvement}.
}
\end{remark}

\section[Asymptotic approximation of logGamma(z+1/2)]%
{Asymptotic approximation of $\logGamma(z+\half)$}
							\label{sec:Gamma2}

In this section we deduce, from the results of \S\ref{sec:Gamma1},
an asymptotic series for $\logGamma(z+\half)$ 
in descending odd powers of $z$.  The series 
was given by Gauss \cite[Art.\ 29]{Gauss-v3};
by using the results of \S\ref{sec:Gamma1} we obtain new
error bounds for $z\in\Hplanestar$.

Replacing $z$ by $2z$ in~\eqref{eq:Stirling} and then
subtracting~\eqref{eq:Stirling} gives
\begin{equation}		\label{eq:G2G}
\logGamma(2z)-\logGamma(z)
= z\log z + (2\log 2 - 1)z - \half\log 2 +
\sum_{j=1}^{k-1} \That_j(z) + \Rhat_k(z),
\end{equation}
where $\That_j(z) = T_j(2z)-T_j(z)$ and $\Rhat_k(z) = R_k(2z)-R_k(z)$.
More explicitly, using~\cite[(8.1.12)]{Olver} for $B_{2j}(\half)$,
we have
\begin{equation}		\label{eq:That}
\That_j(z) = -(1-2^{1-2j})T_j(z) =
 -\frac{(1-2^{1-2j})B_{2j}}{2j(2j-1)z^{2j-1}}
= \frac{B_{2j}(\half)}{2j(2j-1)z^{2j-1}}\,\raisedot
\end{equation}
Also, $\Rhat_k(z) = \That_k(z) + \Rhat_{k+1}(z)$, where
\begin{equation}		\label{eq:Rhat}
\Rhat_{k+1}(z) = -\int_0^\infty \frac{2^{1-2k}B_{2k}(\{2u\}) - B_{2k}(\{u\})}
	{2k(u+z)^{2k}}\dup u.
\end{equation}

Using the duplication formula
$\Gamma(z+\half) = 2^{1-2z}\pi^{1/2}\Gamma(2z)/\Gamma(z)$,
eqn.~\eqref{eq:G2G} immediately gives Gauss's asymptotic expansion
of $\logGamma(z+\half)$:
\begin{equation}		\label{eq:G2G2}
\logGamma(z+\half) = z\log z - z + \half\log(2\pi) +
\sum_{j=1}^{k-1} \That_j(z) + \Rhat_k(z).
\end{equation}

The following lemma enables us to simplify the ``kernel'' function
appearing in the integral~\eqref{eq:Rhat}.
\begin{lemma}					\label{lem:BKhalf}
For $k \ge 1$ and all real $u$,
\[
2^{1-2k}B_{2k}(\{2u\}) - B_{2k}(\{u\}) = B_{2k}(\{u+\half\}).
\]
\end{lemma}
\begin{proof}
This follows from the known identities~\cite[(23.1.8) and (23.1.10)]{AS}
\[B_{2k}(u) = B_{2k}(1-u)\]
and
\[
2^{1-2k}B_{2k}(2u) - B_{2k}(u) = B_{2k}(u+\half).
\]
\end{proof}
\noindent Using Lemma~\ref{lem:BKhalf}, we see from~\eqref{eq:Rhat} that
\begin{equation}		\label{eq:Rhatbd2}
\Rhat_{k+1}(z) = -\int_0^\infty\frac{B_{2k}(\{u+\half\})}
	{2k(u+z)^{2k}}\dup u.
\end{equation}

We can now prove an analogue of Theorem~\ref{thm:right_half_bd}.
The upper bound on $|\Rhat_k(z)|$ is the same as the bound
that we obtained for $|R_k(z)|$, but the bound on $|\Rhat_k(z)/\That_k(z)|$
is larger than the bound on $|R_k(z)/T_k(z)|$ by a factor
$\eta_k = 1/(1-2^{1-2k}) \le 2$.

\begin{theorem}			\label{thm:right_half_bd_a}
If $z\in\Hplanestar$
and $\Rhat_k(z)$ is defined by
eqn.~\eqref{eq:G2G2}, then
\begin{equation}		\label{eq:right_bd1a}
\left|\frac{\Rhat_{k+1}(z)}{\That_k(z)}\right| \le
	\eta_k\frac{\pi^{1/2}\Gamma(k+\half)}{\Gamma(k)}
	\,\raisedot
\end{equation}
\end{theorem}
\begin{proof}
This is almost identical to the proof of
Theorem~\ref{thm:right_half_bd}, the only
difference being that we use~\eqref{eq:Rhatbd2} to bound
$\Rhat_{k+1}(z)$ instead of~\eqref{eq:StirlingR2} to bound $R_{k+1}(z)$.
This increases the bound by a factor
$\eta_k=|T_k(z)/\That_k(z)|$.
\end{proof}

\begin{corollary}		\label{cor:cor1hat}
Under the conditions of Theorem~$\ref{thm:right_half_bd_a}$, 
we have
\begin{equation}		\label{eq:corbd1a}
\left|\frac{\Rhat_{k+1}(z)}{\That_k(z)}\right| < \eta_k\sqrt{\pi k}.
\end{equation}
\end{corollary}

\begin{remark}			\label{remark:k3z1}
{\rm
The factor $\eta_k$ in Corollary~\ref{cor:cor1hat} can be omitted
if $k \ge 3$ or $|z| \ge 1$.  A proof is given in an earlier version of this
paper.\footnote{See arXiv:1609.03682v1, proof of Corollary 3.}
}
\end{remark}

\begin{theorem}			\label{thm:iterative2}
If $z\in\Hplanestar$, $\Rhat_k(z)$ is defined by
eqn.~\eqref{eq:G2G2}, and $k \le |z|$, then
\begin{equation}                \label{eq:corbd3c}
\left|\frac{\Rhat_{k+1}(z)}{\That_k(z)}\right|
 < \eta_k\,\frac{(k/|z|)^2}{\pi^2-1}
\end{equation} 
and
\begin{equation}                \label{eq:corbd3d}
\left|\frac{\Rhat_{k}(z)}{\That_k(z)}\right|
 < 1 + \eta_k\,\frac{(k/|z|)^2}{\pi^2-1}\,\raisedot
\end{equation} 
\end{theorem}
\begin{proof}
This is the same as the proof of Theorem~\ref{thm:iterative},
except that we have to allow for the additional factor $\eta_k$ that
arises because the errors are normalised by $\That_k(z)$
instead of by $T_k(z)$.
\end{proof}

\begin{remark}			\label{remark:eta_ck}
{\rm
By a small modification of Lemma~\ref{lem:ck}, if $k \le |z|$ then
\[
|\Rtilde_{k+1}(z)/\Ttilde_k(z)| \le \eta_k c_k (k/|z|)^2.
\]
}
\end{remark}

\pagebreak[3]
\section{The Riemann-Siegel theta function} \label{sec:RS-theta}

In this section we
consider the Riemann-Siegel theta function $\vth(t)$
defined by~\eqref{eq:RS-theta-1}.
Lemma~\ref{lem:RS-theta-2} gives an equivalent expression for $\vth(t)$
that is better for our purposes than the 
definition.
\begin{lemma}			\label{lem:RS-theta-2}
For all $t\in\R$, 
\begin{equation}		\label{eq:RS-theta-2}
\vth(t) = \half\arg\Gamma\!\left(it+\half\right)
	        - \half t\log(2\pi)
		- \textstyle\frac{\pi}{8}
		+ \half\arctan\left(e^{-\pi t}\right).
\end{equation}
\end{lemma}
\begin{proof}
The reflection formula $\Gamma(s)\Gamma(1-s) = \pi/\sin(\pi s)$
with $s = \frac{it}{2}+\frac14$ gives
\begin{equation}		\label{eq:alt1}
\Gamma\left(\textstyle\frac{it}{2}+\frac14\right)
\Gamma\left(\textstyle-\frac{it}{2}+\frac34\right) = 
\frac{\pi}{\sin \pi(\frac{it}{2}+\frac14)}\,\raisecomma
\end{equation}
and the duplication formula
$\Gamma(s)\Gamma(s+\half) = 2^{1-2s}\pi^{1/2}\Gamma(2s)$ gives
\begin{equation}		\label{eq:alt2}
\Gamma\left(\textstyle\frac{it}{2}+\frac14\right)
\Gamma\left(\textstyle\frac{it}{2}+\frac34\right) = 
2^{1/2-it}\pi^{1/2}\Gamma(it+\half).
\end{equation}
Multiplying~\eqref{eq:alt1} and~\eqref{eq:alt2} gives
\[
\Gamma(\textstyle\frac{it}{2}+\frac14)^2\,
|\Gamma(\textstyle\frac{it}{2}+\frac34)|^2 =
\displaystyle
\frac{2^{1/2-it}\pi^{3/2}\Gamma(it+\half)}
{\sin \pi\!\left(\frac{it}{2}+\frac14\right)}
\,\raisedot
\]
Taking the argument of each side and simplifying, using the fact that
\[
\arctan\left(\frac{1-e^{-\pi t}}{1+e^{-\pi t}}\right) = \frac{\pi}{4} 
	- \arctan\left(e^{-\pi t}\right),
\]
proves the lemma.
\end{proof}

Using the representation of $\vth(t)$ given in Lemma~\ref{lem:RS-theta-2},
and the results of \S\ref{sec:Gamma2}, we obtain an asymptotic approximation
of $\vth(t)$ together with error bounds.
This is summarised in
Theorems~\ref{thm:RS-theta-approx}--\ref{thm:RS-theta-approx_bds}. 
As far as we are aware, this is the first time that
a rigorous error bound applicable for all $k\ge 1$ and all real $t > 0$ has
been given. Most authors seem to restrict themselves to small $k$ and
sufficiently large~$t$.
For example, Edwards~\cite[(2) in \S6.5]{Edwards} takes
$k=2$ and $t$ ``large'';
Gabcke~\cite[Satz 4.2.3(d)]{Gabcke} takes $k=4$ and $t \ge 10$.

\pagebreak[3]

\begin{theorem}			\label{thm:RS-theta-approx}
For all
real $t > 0$, 
\begin{equation}		\label{eq:RS-theta-approx-bd}
\vth(t) = \frac{t}{2} \log\left(\frac{t}{2\pi e}\right)
 - \frac{\pi}{8} + \frac{\arctan\left(e^{-\pi t}\right)}{2}
 + \sum_{j=1}^{k-1}\Ttilde_j(t) + \Rtilde_{k}(t),
\end{equation}
where 
\begin{equation}
\Ttilde_j(t) := \half|\That_j(t)|
 = \frac{|B_{2j}(\half)|}{4j(2j-1)t^{2j-1}}
\end{equation}
and
\begin{equation}		\label{eq:Rtilde_defn}
\Rtilde_{k}(t) := \Im\left(\half\Rhat_{k}(it)\right).
\end{equation}
\end{theorem}
\begin{proof}
{From} Lemma~\ref{lem:RS-theta-2},
\[
2\vth(t) = \Im\left(\logGamma(it+\half)\right)
		- t\log(2\pi)
		- \pi/4
	 	+ \arctan\left(e^{-\pi t}\right).
\]
Using~\eqref{eq:G2G2} with $z=it$ for the $\logGamma(it+\half)$ term, 
we obtain
\begin{align*}
2\vth(t) =&\; \Im\left(it\log(it) - it 
		+ \sum_{j=1}^{k-1}\That_j(it)
		+ \Rhat_k(it)\right)
		- t\log(2\pi)\\
	  &\;\hspace*{19em}   
		- \pi/4
	 	+ \arctan\left(e^{-\pi t}\right)
\end{align*}
Since $B_{2j} = (-1)^{j-1}|B_{2j}|$ and
$B_{2j}(\half) = -(1-2^{1-2j})B_{2j}$,
we see from~\eqref{eq:That} that $\Im(\That_j(it)) = |\That_j(t)|$.
Also, $\Im(it\log i) = \Im(it \cdot i\pi/2) = 0$. Thus,
\begin{align*}
2\vth(t) =&\; t\log t - t
	 + \sum_{j=1}^{k-1}|\That_j(t)|
	 + \Im(\Rhat_{k}(it))
	   - t\log(2\pi)
		- \pi/4
	 	+ \arctan\left(e^{-\pi t}\right)\\
	=&\; t\log\left(\frac{t}{2\pi e}\right) - \displaystyle\frac{\pi}{4}
	+ \arctan\left(e^{-\pi t}\right)
	+ 2\displaystyle\sum_{j=1}^{k-1}\Ttilde_j(t) + 2\Rtilde_k(t).
\end{align*}
Thus, the result~\eqref{eq:RS-theta-approx-bd} follows.
\end{proof}

\begin{remark}			\label{remark:Edwards}
{\rm
The first few terms of the asymptotic expansion for
$\vth(t)$ are derived
in a different manner by Edwards~\cite[\S6.5]{Edwards};
his method does not easily lead to an
expression for the general term 
or to an error bound valid for all~$k$.
}
\end{remark}

\begin{lemma}
For all real $t > 0$,
\begin{equation}		\label{eq:rtilde1}
\Rtilde_1(t) = \Im\left(\int_0^\infty
	\frac{B_2(\half)-B_2(\{u+\half\})}{4(u+it)^2}\dup u\right)
\end{equation}
and 
\begin{equation}		\label{eq:Rtildekplus}
\Rtilde_{k+1}(t) 
    = \Im \left(-\int_0^\infty
      \frac{B_{2k}(\{u+\half\})}{4k(u+it)^{2k}}\dup u \right).
\end{equation}
\end{lemma}
\begin{proof}
Eqn.~\eqref{eq:Rtildekplus} follows from~\eqref{eq:Rhatbd2} and the
definition~\eqref{eq:Rtilde_defn} of $\Rtilde_k(t)$.
For~\eqref{eq:rtilde1}
we use $\Rtilde_1(t) = \Ttilde_1(t) + \Rtilde_2(t)$,
where $\Rtilde_2(t)$ is given by~\eqref{eq:Rtildekplus} with $k=1$.
\end{proof}

\pagebreak[3]

\begin{theorem}			\label{thm:RS-theta-approx_bds}
If $t$ and $\Rtilde_k(t)$ are as in Theorem~$\ref{thm:RS-theta-approx}$,
then
\begin{equation}		\label{eq:RS-theta-Rk-bd}
|\Rtilde_{k+1}(t)| \le
	\frac{\pi^{1/2}\,\Gamma(k-\half)\,|B_{2k}|}{8\,k!\,t^{2k-1}}
	\,\raisedot
\end{equation}
\end{theorem}
\begin{proof}
We use Theorem~\ref{thm:right_half_bd_a} and~\eqref{eq:That}
to bound
$\Rtilde_{k+1}(t) = \half\Im(\Rhat_{k+1}(it))$.\\
(Note that the $\eta_k$ factor in Theorem~\ref{thm:right_half_bd_a} cancels
a factor in~\eqref{eq:That}.)
\end{proof}

\begin{remark}		\label{remark:nonzero_real_part}
{\rm
{From}~\eqref{eq:G2G2}, using the fact that $\Re(\That_j(it)) = 0$,
we have
\begin{align*}
\Re(\Rhat_k(it))
=&\; \Re\left(\logGamma(it+\half) - it\log(it) + it - \half\log(2\pi)\right)\\
=&\; \log|\Gamma(it+\half)| + \textstyle\frac{\pi t}{2} - \half\log(2\pi)\\
=&\; \half\log\left(\frac{\pi}{\cosh \pi t}\right) + 
	\textstyle\frac{\pi t}{2} - \half\log(2\pi)
	\;\;\text{(using \cite[(6.1.30)]{AS})}\\
=&\; -\half\log\left(1+e^{-2\pi t}\right)
= -\half e^{-2\pi t} + O(e^{-4\pi t}),
\end{align*}
so $\Re(\Rhat_k(it))$ is exponentially small, but nonzero.
Thus $|\Rtilde_k(t)| < \half|\Rhat_k(it)|$,
and it follows that the inequality~\eqref{eq:RS-theta-Rk-bd} is strict.
}
\end{remark}

\begin{corollary}		\label{cor:cor3}
If $t>0$ then
\begin{equation}		\label{eq:RS-theta-ratio}
\left|\frac{\Rtilde_{k+1}(t)}{\Ttilde_k(t)}\right| <
 \eta_k\,\sqrt{\pi k}.
\end{equation}
\end{corollary}
\vspace*{-10pt}
\begin{proof}
This follows from Corollary~\ref{cor:cor1hat} with $z = it$.
\end{proof}
\pagebreak[3]

\begin{remark}			\label{remark:k3z1b}
{\rm
The factor $\eta_k$ in Corollary~\ref{cor:cor3} can be omitted
if $k \ge 3$ or $t \ge 1$ (see Remark~\ref{remark:k3z1}).
}
\end{remark}

\begin{corollary}		\label{cor:cor4}
If $t\ge k > 0$, then
\[
\left|\frac{\Rtilde_{k+1}(t)}{\Ttilde_k(t)}\right| <
  \eta_k\,
  \frac{(k/t)^2}{\pi^2-1}\,\raisedot
\]
\end{corollary}
\begin{proof}
This follows from Theorem~\ref{thm:iterative2} with $z = it$.
\end{proof}
\begin{remark}			\label{remark:k3z1c}
{\rm
The factor $\eta_k$ in Corollary~\ref{cor:cor4} can be omitted
if $k \ge 3$. This follows for sufficiently large $k$ from a slight
modification of the proof of Theorem~\ref{thm:iterative2},
and for small $k$ from the observation
that $\eta_k c_k < 1/(\pi^2-1)$ for $k \ge 3$
(see Remark~\ref{remark:eta_ck} and Table~{\Tableck}).
If $1\le k \le 2$ we can use the bound
$\eta_k c_k (k/t)^2$ that follows from Remark~\ref{remark:eta_ck}.
}
\end{remark}

In the literature, the asymptotic
approximation~\eqref{eq:RS-theta-approx-bd} always seems to be stated
without the exponentially-small arctan term.  See, for example,
Edwards~\cite[(1) on pg.~120]{Edwards},
Gabcke~\cite[Satz~4.2.3(c)]{Gabcke}, 
and Lehmer~\cite[(5) on pg.~104]{Lehmer2}.
The arctan term appears in some related formulas,
such as Gram~\cite[(7) on pg.~300]{Gram}
and Gabcke~\cite[Satz~4.2.3(a)]{Gabcke}.
See also the discussion in Berry~\cite[\S4]{Berry95}.

It is valid to omit the arctan term
if all we want is an asymptotic
series in the sense of Poincar\'e (see Olver~\cite[\S1.7.3]{Olver}).
However, it is not desirable if we want to minimise the error
in the approximation. If we omit the arctan term, then the upper bounds
on $|\Rtilde_k(t)|$ have to be increased accordingly.
Since $\arctan(e^{-\pi t}) < e^{-\pi t}$ for $t \ge 0$,
it is sufficient to add $\half e^{-\pi t}$ to the bound
on $|\Rtilde_{k+1}(t)|$ 
in~\eqref{eq:RS-theta-Rk-bd}.
The bound of Corollary~\ref{cor:cor3} can be replaced by
\begin{equation}		\label{eq:bd_no_arctan}
|\Rtilde_{k+1}(t)| < \eta_k\,\sqrt{\pi k}\,\Ttilde_k(t) + 
	\half e^{-\pi t}.
\end{equation}

Of course, $\half e^{-\pi t}$
is negligible if $t$ is large, e.g.~when searching for high zeros
of $\zeta(s)$ on the critical line. When $t$ is not so large, the arctan
term may be significant. We discuss this in the next section.

\begin{remark}			\label{remark:Stokes}
{\rm
Other situations where an exponentially small contribution is
significant are mentioned by Watson~\cite[\S\S7.22--7.23]{Watson},
in connection with the Stokes phenomenon \cite{Meyer,Olde-Daalhuis}
and the asymptotic expansions
of the Bessel functions $J_\nu(z)$ and  $I_\nu(z)$.
An example that is similar to ours, but somewhat simpler, was given by
Olver~\cite{Olver64}, and is discussed by
Meyer~\cite[Appendix]{Meyer}.
}
\end{remark}

\section{Attainable accuracy}		\label{sec:accuracy}

In this section we consider the accuracy of
the asymptotic expansion of 
$\vth(t)$
if $t$ is fixed and we choose (close to) the optimal number of terms
to sum.

\comment{
MOVE NEXT TWO PARAGRAPHS

An interesting question is ``for given $t\ge 1$, what is the minimum error
that can be guaranteed when using the asymptotic expansion of
Theorem~\ref{thm:RS-theta-approx} to approximate $\vth(t)$?''.

Suppose we use the first $k$ terms in the sum $\sum_{j}\Ttilde_j(t)$,
so the error is $\Rtilde_{k+1}(t)$. 
If we bound $\Rtilde_{k+1}(t)$ using~\eqref{eq:RS-theta-ratio},
i.e.~using $|\Rtilde_{k+1}(t)| < (\pi k)^{1/2}\,\Ttilde_k(t)$, then
we should choose $k$ to minimise $(\pi k)^{1/2}\,\Ttilde_k(t)$.
However, this choice is artificial, since it depends on 
\emph{our bound on the error} rather than on the \emph{true error}.
Numerical evidence indicates that the true error is closer to
$\Ttilde_k(t)$ than to $(\pi k)^{1/2}\,\Ttilde_k(t)$ if
$k$ is chosen to minimise either $(\pi k)^{1/2}\,\Ttilde_k(t)$
or $\Ttilde_k(t)$ (there is little difference between these
two choices since
$k^{1/2}$ varies relatively slowly near the respective minima).
Thus, for the sake of simplicity, we choose $k\in\Nstar$ to minimise
$\Ttilde_k(t)$. Denote this value of $k$ by $k_{\rm min}$.
} 

Assume that $t$ is fixed and positive.
The terms $\Ttilde_k(t)$ initially decrease (unless
$t \le \sqrt{7/120} \approx 0.2415$),
but eventually increase in
value, so it is of interest to determine the index of a minimal term.
Define
\[k_{\rm min} = k_{\rm min}(t) := \min \{k\ge 1: \Ttilde_k(t) \le \Ttilde_{k+1}(t)\}\]
and
\[\Ttilde_{\rm min}(t) := \Ttilde_{k_{\rm min}}(t).\]
Lemma~\ref{lemma:unimodal} shows that, for all $t > 0$, 
the sequence of terms
$(\Ttilde_k(t))_{k\ge 1}$ is unimodal, and that
$\Ttilde_{\rm min}(t)$ is a minimal term.

\begin{lemma}		\label{lemma:unimodal}
Fix $t > 0$. Then\\
$(1)$ for $1 \le k < k_{\rm min}(t)$, $\Ttilde_k(t) > \Ttilde_{k+1}(t) > 0$;\\
$(2)$ for $k = k_{\rm min}(t)$, $0 < \Ttilde_k(t) \le \Ttilde_{k+1}(t)$;\\
$(3)$ for $k > k_{\rm min}(t)$, $0 < \Ttilde_k(t) < \Ttilde_{k+1}(t)$;\\
$(4)$ $\Ttilde_{\rm min}(t) = \min_{k\ge 1} \Ttilde_k(t)$.
\end{lemma}
\begin{proof}[Proof (sketch)]
We observe that, for all $k \in \Nstar$,
\[R(k) := \displaystyle\frac{\Ttilde_{k+1}(t)/\Ttilde_{k+2}(t)}
      {\Ttilde_{k}(t)/\Ttilde_{k+1}(t)}\]
is independent of $t$, and can be shown to lie in the interval
$(0,1)$.  (This is clear for large $k$, since
\begin{equation*}
R(k) = \frac{k(2k-1)}{(k+1)(2k+1)}\left(1+O(4^{-k})\right),
\end{equation*}
and can be verified by a 
numerical computation for small~$k$.)
Thus
\[
\frac{\Ttilde_{k+1}(t)}{\Ttilde_{k+2}(t)} <
\frac{\Ttilde_{k}(t)}{\Ttilde_{k+1}(t)}\,\raisedot
\]
The inequalities $(1)$--$(3)$ of the lemma now follow easily,
and the equality $(4)$ follows from $(1)$--$(3)$.
\end{proof}

\begin{lemma}		\label{lemma:kmin}
For large positive $t\in\R$,
\[
k_{\rm min}(t) = \pi t + O(1)\]
and, if $k = \pi t + O(1)$, then
\[
\Ttilde_k(t) =
 \frac{e^{-2\pi t}}{2\pi\sqrt{t}}\left(1+O\left(\frac{1}{t}\right)\right).
\]
\end{lemma}
\begin{proof}[Proof (sketch)]
{From} $|B_{2k}| = 2(2k)!\,\zeta(2k)/(2\pi)^{2k}$ we obtain
\begin{equation}		\label{eq:Ttilde_ratio}
\frac{\Ttilde_k(t)}{\Ttilde_{k+1}(t)} = 
 \frac{2k(2k-1)}{4\pi^2 t^2}\left(1+O(4^{-k})\right).
\end{equation}
Thus, $k_{\rm min} = \pi t + O(1)$, where the $O(1)$ term covers the
$1+O(4^{-k})$ factor and the effect of rounding to the nearest integer.

The estimate of $\Ttilde_{k}(t)$ follows from Stirling's approximation.
Write\linebreak
$k = \pi t/(1+\varepsilon)$, so $\varepsilon = O(1/t)$.
Then
\begin{align*}
\Ttilde_k(t) =&\; \frac{(1-2^{1-2k})\,\zeta(2k)\,(2k)!}
	{2k(2k-1)\,(2\pi)^{2k}\,t^{2k-1}}\\
	=&\; \frac{t}{4k^2} \left(\frac{2k}{e}\right)^{2k}
	\!\!\!\frac{\sqrt{4k\pi}}{(k(1+\varepsilon))^{2k}}\,
	(1+O(\varepsilon))\\
	=&\; \frac{e^{-2k-2k\varepsilon}}{2\pi \sqrt{t}}\,(1+O(\varepsilon))\\
	=&\; \frac{e^{-2\pi t}}{2\pi \sqrt{t}}\,(1+O(\varepsilon)).
\end{align*}
\end{proof}
\begin{remark}
{\rm
If we minimise $({\pi k})^{1/2}\,\Ttilde_k(t)$ instead of $\Ttilde_k(t)$,
the minimum is still at $k = \pi t + O(1)$. The difference between the indices
of the two minima can be subsumed by the $O(1)$ term.
}
\end{remark}
\begin{corollary}		\label{cor:Rkplusbd}
If $k = \pi t + O(1)$, then
$|\Rtilde_{k+1}(t)| < \half e^{-2\pi t}(1+O(1/t))$.
\end{corollary}
\begin{proof}
The result follows from~\eqref{eq:RS-theta-ratio} and the second
half of Lemma~\ref{lemma:kmin}.
\end{proof}

{From} Lemma~\ref{lemma:kmin} and Corollary~\ref{cor:Rkplusbd},
we can guarantee an error not exceeding $\half e^{-2\pi t}(1+O(1/t))$
by taking $k_{\rm min}(t) = \pi t + O(1)$ terms in the approximation
\begin{equation}		\label{eq:RS-theta-approx-bdagain}
\vth(t) \approx \frac{t}{2} \log\left(\frac{t}{2\pi e}\right)
 - \frac{\pi}{8} + \frac{\arctan\left(e^{-\pi t}\right)}{2}
 + \sum_{j=1}^{k_{\rm min}(t)}\Ttilde_j(t).
\end{equation}
On the other hand, if we use the ``standard''
approximation
\begin{equation}		\label{eq:RS-theta-approx-bdagain2}
\vth(t) \approx \frac{t}{2} \log\left(\frac{t}{2\pi e}\right)
 - \frac{\pi}{8} + \sum_{j=1}^{k_{\rm min}(t)}\Ttilde_j(t),
\end{equation}
we can only guarantee an error not exceeding
$\half e^{-\pi t} + O(e^{-2\pi t})$.
Thus, the $\arctan$ term is numerically significant, even though
it is asymptotically smaller than any term $\Ttilde_j(t)$.
This is illustrated by
Table \Tableabcd, where we give, for various $t\in[1,100]$,
$k_{\rm min}(t)$ and\\[-20pt]
\begin{itemize}
\item[$A:$] the error in the standard approximation
\eqref{eq:RS-theta-approx-bdagain2} after taking $k_{\rm min}(t)$ terms,
normalised by the smallest term 
$\Ttilde_{\rm min}(t) \approx e^{-2\pi t}/(2\pi t^{1/2})$;\\[-20pt]
\item[$B:$] the error bound of \eqref{eq:RS-theta-ratio}
	(this is already normalised)
	;\\[-20pt]
\item[$C:$] the error in the approximation
\eqref{eq:RS-theta-approx-bdagain},
normalised by the smallest term,
i.e.\ $\Rtilde_{k+1}(t)/\Ttilde_k(t)$ for $k = k_{\rm min}(t)$;\\[-20pt]
\item[$D:$] the error in the empirically improved approximation
\begin{align}
\vth(t) \approx&\; \frac{t}{2} \log\left(\frac{t}{2\pi e}\right)
 - \frac{\pi}{8} + \frac{\arctan\left(e^{-\pi t}\right)}{2}
	\nonumber\\
 &\; + \sum_{j=1}^{k_{\rm min}(t)}\Ttilde_j(t)
 + \left(\pi t - k_{\rm min}(t) + \frac{1}{12}\right)\Ttilde_{\rm min}(t),
	\label{eq:empirical}
\end{align}
normalised by $\Ttilde_{\rm min}(t)$, as for columns $A$ and $C$. 
\end{itemize}

It can be seen that $k_{\rm min}(t)$ is usually
$\lfloor \pi t + 5/4 \rfloor$.
This is as expected from~\eqref{eq:Ttilde_ratio}.
The normalised value $A$ is 
approximately $\pi t^{1/2}\exp(\pi t)$, which is large
because $\Ttilde_{\rm min}(t)$, given by Lemma~\ref{lemma:kmin},
is much smaller than the error, which is about $\half\exp(-\pi t)$. 

Column $B$ gives upper bounds on the absolute values
of the entries in column $C$~--
it is clear that the upper bounds are conservative
(although necessarily so, by the discussion near the end of
\S\ref{sec:Gamma1}).

It can be observed that the entries in column $C$ are negative.
This suggests that we would be better off truncating
the sum after $k_{\rm min}-1$ terms instead of $k_{\rm min}$ terms (which would
have the effect of adding $1$ to the entries in column $C$). However, a
much better approximation is obtained by adding a ``correction term''
\begin{equation*}		
\left(\pi t - k_{\rm min}(t) + \frac{1}{12}\right)\Ttilde_{\rm min}(t)
\end{equation*}
as in~\eqref{eq:empirical}. 
The motivation for the correction term is to smooth out the sawtooth nature
of approximation $C$, which has jumps at the values of $t$ where
$k_{\rm min}(t)$ changes.  This explains the addition of
$(\pi t - k_{\rm min}(t)+c)\,\Ttilde_{\rm min}(t)$, where $c$ is an arbitrary
constant. 
Column $D$ gives numerical evidence
for a constant close to $\frac{1}{12}$.
We do not have a theoretical
explanation for the value of this constant,
although it is clearly related to the asymptotic location
of the positive zero(s) of
the function $\Rtilde_{k+1}(t)$ given
by~\eqref{eq:Rtildekplus}.
It may be relevant that, for large $k$,
$B_{2k}(u+\half)$ behaves like a scaled version
of $\cos(2\pi u)$: see Dilcher~\cite[Theorem~1]{Dilcher}.

\begin{table}[ht]
\begin{center}
\begin{tabular}{cccccc}
$t$  & $k_{\rm min}$ & $A$ 	       & $B$ 	& $C$ 	  & $D$ \\
\hline
$1$  & $4$	 & $7.2\tm{1}$ & $3.57$ & $-0.79$ & $-1.1\tm{-2}$\\
$2$  & $7$	 & $2.4\tm{3}$ & $4.69$ & $-0.63$ & $+2.4\tm{-4}$\\
$5$  & $16$	 & $4.6\tm{7}$ & $7.09$ & $-0.21$ & $+2.8\tm{-3}$\\
$10$ & $32$	 & $4.4\tm{14}$& $10.0$ & $-0.50$ & $+8.3\tm{-4}$\\
$20$ & $64$	 & $2.7\tm{28}$& $14.2$ & $-1.08$ & $+8.3\tm{-5}$\\
$50$ & $158$	 & $3.7\tm{69}$& $22.3$ & $-0.84$ & $-1.5\tm{-4}$\\
$100$&$315$	 & $8.6\tm{137}$&$31.5$ & $-0.76$ & $-5.2\tm{-5}$\\
\hline
\end{tabular}
\caption{Normalised errors -- see text for $A, B, C,  D$.}
\end{center}
\vspace*{-10pt}
\end{table}

\subsection*{Acknowledgement}

The author was supported in part by Australian Research Council grant
DP140101417.

\pagebreak[3]

\end{document}